\documentclass[english,12pt]{smfart}
  \usepackage[utf8]{inputenc}
  \usepackage[T1]{fontenc}

  \usepackage{amssymb,stmaryrd}
  \usepackage{smfthm} 
  \usepackage{graphicx}
  \usepackage{color}
  \usepackage{mathrsfs}

\newcommand{\cw}{\mathop {\mathrm{cw}}}

\newcommand{\mR}{\mathbb{R}}

\newtheorem*{theo*}{Theorem}
\newtheorem*{coro*}{Corollary}

\author{Beno\^{\i}t Kloeckner}
\title{Cutwidth and degeneracy of graphs}

\begin{document}

\begin{abstract}
We prove an inequality involving the degeneracy, the cutwidth and the sparsity of
graphs. It implies a quadratic lower bound on the cutwidth in terms of the
degeneracy for all graphs and an improvement of it for clique-free graphs.
\end{abstract}

\maketitle

\section{Introduction}

The starting point of the author's interest in cutwidth is a lecture by
Misha Gromov at the ``Glimpses of Geometry'' conference held in Lyon in May
2008. During this lecture, Gromov introduced a concept similar to
cutwidth in the realm of topology of manifolds. Although this paper
will stick to combinatorics of graphs, let us give an idea of this important 
motivation to this work.

Gromov's question was the following: given a manifold $X$
and a continuous map $F:X\to\mathbb{R}$, how can one relate the topological
complexity of $X$ to the maximum topological complexity of the level sets of $F$ ?
It turns out that the answer depends heavily on the dimension of $X$. If $X$
is an orientable surface of arbitrarily high genus, it is always possible
to design $F$ so that its levels are a point, a circle, a couple of circles, a figure eight
or empty (see figure \ref{fig:surface}). 
The complexity of $X$ is therefore not bounded by the complexity of level sets of $F$.
The picture gets different in some higher dimensions but it is not our purpose to
detail this here.

This question also makes sense for polyhedron, and  in \cite{Gromov} Gromov proves several results
in this setting. That paper raises many questions that could be of great interest
to combinatoricians interested in cutwidth.

\begin{figure}[htp]\begin{center}
\includegraphics{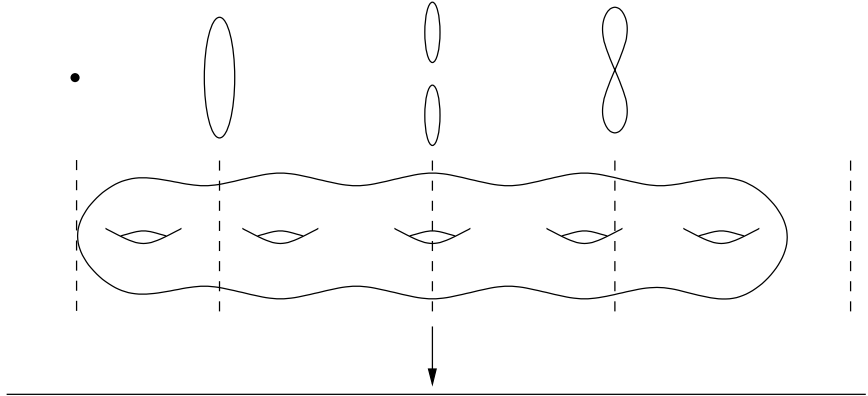}
\caption{A map from a surface of arbitrarily high genus to the line,
 with level sets of bounded complexity.}\label{fig:surface}
\end{center}\end{figure}

\subsection{Cutwidth}
The complexity of a level set makes sense in graphs: if a simple graph $G$ is identified with its topological
realization, one can consider the infimum, over all continuous maps $f:G\to\mR$, of the maximal 
multiplicity (that is, number of inverse images) of points of $\mR$. This turns out to be
exactly the \emph{cutwidth} of $G$, usually defined as follows.

Given a linear ordering $\mathcal{O}=(x_1<\cdots<x_n)$ of the vertices, one defines the
cutwidth of the ordering as
$$\cw(G,\mathcal{O})=\max_{1\leqslant i\leqslant n}\#\{(uv)\in E\,|\,u\leqslant x_i<v\}$$
Then, the cutwidth of $G$ is defined as
$$\cw(G)=\min_{\mathcal{O}} \cw(G,\mathcal{O})$$
where $\mathcal{O}$ runs over all linear orderings of $V$.

Now, given an ordering 
$\mathcal{O}$ of $V$ that is optimal with respect to cutwidth, an optimal continuous map is obtained
by mapping each vertex to the integer corresponding to its rank, and mapping edges monotonically between their
endpoints. Its maximal multiplicity is then the cutwidth of the ordering.
Conversely, it can be shown that any continuous map $G\to\mR$ can be modified into a map
that is one-to-one on $V$ and monotonic on edges without raising its maximal multiplicity.

To show the relevance of this point of view, let us give a very short proof of
the main result of \cite{Chavez-Trapp} (this is independent of the rest of the article).
One defines the \emph{circular cutwidth} of a graph
in the same way than its cutwidth, but replacing the line by a circle (the combinatorial 
definition is in terms of a cyclic ordering of the vertices, plus data determining for each edge
whether it is drawn clockwise or counterclockwise).
\begin{theo}
The circular cutwidth of a tree is equal to its cutwidth.
\end{theo}

\begin{proof}
First, since the line topologically embeds in the plane, it is clear that
the circular cutwidth of any graph cannot exceed its cutwidth.

Consider an optimal continuous map $f:T\to \mathbb{S}^1$ where $T$ is (the topological
realization of) a tree. Let $e:\mathbb{R}\to\mathbb{S}^1$ be the universal covering map.
The path lifting property shows, since $T$ has no cycle, that there is a continuous map 
$\tilde f:T\to \mathbb{R}$ such that $e\circ \tilde f=f$. The inverse image by $\tilde f$
of a point $x$ is contained in the inverse image by $f$ of $e(x)$, so that the cutwidth
of $T$ cannot exceed its circular cutwidth.
\end{proof}

\subsection{Degeneracy}
What we need next is to define the complexity of a graph. There are many invariants that
can play this r{\^o}le; here we use the \emph{degeneracy}, a montonic version
of the minimal degree, defined as follows.

Given an integer $k$, the $k$-core $G_k$ of $G$ is the subgraph obtained by recursively
pruning the vertices of degree strictly less than $k$. The degeneracy $\delta(G)$
of the graph is the largest $k$ such that $G_k$ is not empty. A graph with big degeneracy
is in some sense thick.

One of the features of degeneracy is that it is an upper bound for the chromatic and 
list-chromatic numbers:
$$\chi(G)\leqslant\chi_\ell(G)\leqslant\delta(G)+1.$$
The proof of this is classical, see for example the chapter on five-coloring of planar graphs in
\cite{Aigner-Ziegler}.

\subsection{Sparsity}

In order to get a more interesting inequality, we need to involve another invariant
of graphs. We shall use a uniform variant of sparsity, which can be controlled
for clique-free graphs, and will enable us to deduce some kind of expanding property
for $G$.

A graph $G$ on $n$ vertices is said to be $\lambda$\emph{-sparse}
(where $\lambda\geqslant 1)$ if it has at most $n(n-1)/(2\lambda)$ edges.  We shall say
that $G$ is \emph{$(\rho,\lambda)$-uniformly sparse} if all subgraphs of $G$ that
contain at least $\rho n$ vertices are $\lambda$-sparse.

Note that we cannot ask for sparsity of all subgraphs of $G$, since a subgraph
consisting of two adjacent vertices is not $\lambda$-sparse for any $\lambda>1$.

\section{A quadratic inequality}

\subsection{The main result}
We start by a general lower bound on the cutwidth of a graph in terms of
its degeneracy and uniform sparsity.

\begin{theo}\label{theo:main}
For all $(\rho,\lambda)$-uniformly sparse graph $G$ on $n$ vertices we have
\begin{equation}
\cw(G)\geqslant\lceil\rho n\rceil\left(\delta(G)-\frac{\lceil\rho n\rceil-1}{\lambda}\right).\label{eq:main2}
\end{equation}
Moreover if $2n\rho \leqslant\delta(G)\lambda-1$ then
\begin{equation}
\cw(G)>\frac{(\delta(G)\lambda+1)^2}{4\lambda}-\frac1{\lambda}.
\label{eq:main}
\end{equation}
\end{theo}
It may seem strange that in \eqref{eq:main} the cutwidth is bounded from below by an \emph{increasing} function
of the sparsity; this simply translates the fact that when (uniform) sparsity increases,
the degeneracy decreases more than the cutwidth does.

Note that in some classes of graphs, the possibility of choosing $\rho$ enables one to get
a bound that is quadratic in $\delta$ from \eqref{eq:main2} too. As a matter of fact,
\eqref{eq:main} is simply an optimization of \eqref{eq:main2} when $\rho$ can be taken small enough.
Since it can be difficult to prove $(\rho,\lambda)$-sparsity with good constants,
it is not obvious that there is a need for such a general statement. It is mainly
motivated by corollaries \ref{coro:notriangle} and \ref{coro:noclique} on clique-free graphs.

\begin{proof}
We consider a simple graph $G$ on $n$ vertices that is assumed to be $(\rho,\lambda)$-uniformly sparse.
Let $G'$ be the $\delta(G)$-core of $G$: its minimal
degree is $\delta(G)$ and since it is a subgraph of $G$, $\cw(G)\geqslant\cw(G')$.

Let $\mathcal{O}=(x_1<\cdots<x_{n'})$ be a linear ordering of the vertices of $G'$
that minimizes $\cw(G',\mathcal{O})$.
For all $i$ let $n_i=\#\{(uv)\in E'\,|\,u\leqslant x_i<v\}$ and denote by
$G(i)$ the subgraph of $G'$ induced on the vertices $\{x_1,\ldots,x_i\}$.
By assumption, for all $i\geqslant\rho n$, the graph
$G(i)$ has at most $i(i-1)/(2\lambda)$ edges. The total
sum of the degrees in $G'$ of the vertices of $G(i)$ is at least $i\delta(G)$, so that
$$n_i\geqslant i\delta(G)-\frac{i^2-i}{\lambda}.$$
If $2n\rho \leqslant\delta(G)\lambda-1$, we can evaluate this inequality at the optimal point
$i=\lfloor (\delta(G)\lambda+1)/2\rfloor$ since it satisfies $i\geqslant\rho n$.
Letting $\varepsilon=(\delta(G)\lambda+1)/2-\lfloor (\delta(G)\lambda+1)/2\rfloor$ we then get
\begin{eqnarray*}
\cw(G) &\geqslant& \left(\frac{\delta(G)\lambda+1}2-\varepsilon\right)
                   \left(\delta(G)-\frac{\left(\frac{\delta(G)\lambda+1}2-\varepsilon\right)-1}\lambda\right) \\
       &\geqslant& \frac{\delta(G)\lambda+1-2\varepsilon}{2}
                   \left(\frac{\delta(G)\lambda+2\varepsilon+1}{2\lambda}\right)\\
       &=& \frac{(\delta(G)\lambda+1)^2}{4\lambda}-\frac{\varepsilon^2}{\lambda}
\end{eqnarray*}
which gives the desired inequality since $\varepsilon<1$.

In any case, we can consider the point $i=\lceil\rho n\rceil$ and get
\eqref{eq:main2}.
\end{proof}

\subsection{Application to general graphs}

If we let down any information on $G$, we get the
following.
\begin{coro}\label{coro:general}
For all simple graphs, we have
\begin{equation}
\cw(G)\geqslant \frac14 \delta(G)^2+\frac12\delta(G).
\end{equation}
\end{coro}

\begin{proof}
Since every graph is $(0,1)$-uniformly sparse, from \eqref{eq:main}
we deduce that $\cw(G)>(\delta(G)+1)^2/4-1$. If $\delta(G)$
is odd, then it follows that $\cw(G)\geqslant(\delta(G)+1)^2/4$ but otherwise,
writting $\delta(G)=2k$ we see that $\cw(G)>k^2+k-3/4$ so that 
$\cw(G)\geqslant k^2+k$.
\end{proof}

This is a positive answer to our version of Gromov's question:
all continuous maps from a high-complexity graph to the line
have high multiplicity. Of course, in many cases this estimate is rather poor:
for example trees have degeneracy $1$ and unbounded cutwidth (so that
there is no lower bound of $\delta(G)$ in terms of $\cw(G)$) and hypercubes
have exponential cutwidth but linear degeneracy. However it is sharp for
complete graphs, and a better bound would have to involve
more information on the graph.

As pointed out to me by professors Raspaud and Gravier, the consequence 
in terms of chromatic number is in fact easy to prove directly: 
consider an optimal coloring as a morphism $G\to K_{\chi(G)}$
(which must be onto the edge set), and use that $\cw(K_n)=\lfloor n^2/4 \rfloor$.
However Corollary \ref{coro:general} is stronger in the sense that it applies to
the degeneracy, and in particular implies a bound on the list chromatic number. 

\subsection{The case of clique-free graphs}

We shall deduce the following from Theorem \ref{theo:main}.
\begin{coro}\label{coro:notriangle}
For all simple graph $G$ without triangle, one has
\begin{equation}
\cw(G)\geqslant \frac12 {\delta(G)^2}.
\end{equation}
\end{coro}

\begin{proof}
The main point is to show that triangle-free graphs are somewhat sparse; but
Tur{\'a}n's Theorem \cite{Turan} (see also \cite{Bollobas}) in particular gives
that any graph without triangle on $n$ vertices is $2\frac{n-1}{n}$-sparse.

Now, when all subgraphs of $G$ are triangle-free, we get that
$G$ is $(\rho,2(\rho n-1)/(\rho n))$-uniformly sparse for all $\rho$ such that
$\rho n$ is an integer.
Applying the second part of the main theorem, we get that
\begin{eqnarray*}
\cw(G) &\geqslant& \rho n\left(\delta(G)-\frac{\rho n-1}{2\frac{\rho n -1}{\rho n}}\right)\\
       &\geqslant& \rho n\left(\delta(G)-\frac12\rho n\right)\\
\end{eqnarray*}
and taking $\rho=\delta(G)/n$, the desired inequality follows.
\end{proof}

With the same argument, one can deduce the following from Tur{\'a}n's Theorem.

\begin{coro}\label{coro:noclique}
For all simple graph $G$ without subgraph isomorphic to $K_{k+1}$, one has
\begin{equation}
\cw(G)\geqslant \frac{k}{k-1}\frac{\delta(G)^2}4 -\frac{k-1}{k}.
\end{equation}
\end{coro}

Let us show that this gives an asymptotically sharp result for Tur{\'a}n's graph $\mathrm{Tur}(n,k)$,
defined as the most balanced
complete $k$-partite graph on $n$ vertices. On the one hand Corollary \ref{coro:noclique} gives
\begin{equation}
\cw(\mathrm{Tur}(n,k))\geqslant \frac{k-1}{k}\frac{n^2}{4}-\frac{n}2-\frac{3k}{4(k-1)}
\end{equation}
but on the other hand, one can give an explicit ordering of vertices of $\mathrm{Tur}(n,k)$ with cutwidth
of the same order of magnitude. Indeed, the graph whose vertices
are the integers $\{1,2,\ldots,n\}$ and where two vertices $a,b$ are connected by an edge if and only
if $a\not\equiv b\mod k$ is isomorphic to $\mathrm{Tur}(n,k)$ and endowed with a natural ordering.
The number of edges that cross the vertex $i$ and are issued from any fixed vertex $j<i$ is at most 
$$n-i-\left\lceil\frac{n-i}k\right\rceil\leqslant\frac{k-1}k(n-i)+1$$
so that the total number of edges crossing $i$ is at most
$$c(i)=i\left((n-i)\frac{k-1}k+1\right).$$
Now the function $c$ takes its maximal value at $x=(n+k/(k-1))/2$ so that we get
\begin{equation}
\cw(\mathrm{Tur}(n,k))\leqslant \frac{k-1}k\frac{n^2}4+\frac{n}2+\frac{k}{4(k-1)}.
\end{equation}

This approach also applies to all solved
forbidden subgraph extremal problems, for example to $\mathrm{Tur}(rt,r)$-free
graphs \cite{Erdos-Stone}, see also \cite[Theorem VI.3.1]{Bollobas}.

\bibliographystyle{smfplain}
\bibliography{biblio.bib}

\end{document}